\documentclass[11pt,draft]{article}

\usepackage[a4paper,includeheadfoot,margin=2.54cm]{geometry}
\usepackage[latin2]{inputenc}
\usepackage[english]{babel}
\usepackage{amsthm}
\usepackage{amsmath}
\usepackage{amssymb}
\usepackage[hyphens]{url}

\theoremstyle{plain}
\newtheorem{thm}{Theorem}
\newtheorem{lemma}{Lemma}
\newtheorem{prop}[lemma]{Proposition}

\theoremstyle{definition}
\newtheorem{definition}{Definition}
\newtheorem{problem}{Problem}

\theoremstyle{remark}

\newcommand{\F}{\mathcal{F}}
\newcommand{\N}{\mathbb{N}}
\renewcommand{\S}{\mathcal{S}}
\renewcommand{\P}{\mathbb{P}}
\newcommand{\A}{\mathbf{A}}
\newcommand{\B}{\mathbf{B}}
\renewcommand{\v}{\mathbf{v}}
\DeclareMathOperator{\rank}{rk}
\newcommand{\E}{\mathbb{E}}

\begin{document}
\title{The cross-correlation measure of families of finite binary sequences: limiting distributions and minimal values}
\author{
L\'aszl\'o M\'erai, 
\\
{\small Johann Radon Institute for Computational and Applied Mathematics}\\
{ \small Austrian Academy of Sciences}\\
{ \small Altenbergerstr.\ 69, 4040 Linz, Austria}\\
{ \small e-mail: merai@cs.elte.hu}}
\maketitle

\begin{abstract}
Gyarmati, Mauduit and S\'ark\"ozy introduced the \textit{cross-correlation measure $\Phi_k(G)$ of order $k$} to measure the level of pseudorandom properties of families of finite binary sequences. 
In an earlier paper we estimated the cross-correlation measure of a random 
family of binary sequences. In this paper, we sharpen these earlier 
results by showing that for random families, the cross-correlation measure converges strongly, and so has limiting distribution. We also give sharp bounds to the minimum values of the cross-correlation measure, 
which settles a problem of Gyarmati, Mauduit and S\'ark\"ozy nearly completely.
\end{abstract}


\textit{Key words and phases:} pseudorandom sequences, binary sequence, correlation measure, cross-correlation measure

\let\thefootnote\relax\footnote{The final publication is available at Elsevier via
\url{http://dx.doi.org/10.1016/j.dam.2016.06.024}\\
\copyright \ 2016. This manuscript version is made available under the CC-BY-NC-ND 4.0 license \url{http://creativecommons.org/licenses/by-nc-nd/4.0/} 
}

\section{Introduction}

Recently, in a series of papers the pseudorandomness of \textit{finite binary sequences} $E_N=(e_1,\dots,\allowbreak e_N)\in\{-1,1\}^N$ has been studied. In particular, measures of pseudorandomness have been defined and investigated; see \cite{AKMMR,CMS,measures-survay,measures} and the references therein.

For example, Mauduit and S\'ark\"ozy \cite{measures} introduced the \textit{correlation measure $C_k(E_N)$ of order $k$}  of the binary sequence $E_N$. Namely, for a $k$-tuple $D=(d_1,\dots, d_k)$ with non-negative integers $0\leq d_1 < \dots <d_k<N$ and $M\in\mathbb{N}$ with $M +d_k\leq N$ write
\begin{equation*}
 V_k(E_N,M,D)=\sum_{n=1}^M e_{n+d_1}\dots e_{n+d_{k}}.
\end{equation*}
Then $C_k(E_N)$ is defined as
\begin{equation*}
 C_{k}(E_N)=\max_{M,D}\left| V(E_N,M,D) \right| =\max_{M,D}\left|
\sum_{n=1}^M e_{n+d_1}\dots e_{n+d_{k}} \right|.
\end{equation*}

This measure has been widely studied, see, for example \cite{ACS,AKMMR-min,AKMMR,CMS,gyarmati-correlation-measure,MS-2003,kai}.
In particular, Alon,  Kohayakawa, Mauduit,  Moreira and R\"odl \cite{AKMMR} obtained the typical order of magnitude of $C_k(E_N)$.  They proved that, if $E_N$ is chosen uniformly from $\{-1,+1\}^N$, then for all $0<\varepsilon<1/16$ the probability that
\begin{equation*}
 \frac{2}{5}\sqrt{N \log \binom{N}{k}}<C_k(E_N)< \frac{7}{4}\sqrt{N \log \binom{N}{k}}
\end{equation*}
holds for every integer $2\leq k \leq N/4$ is at least $1-\varepsilon$ if $N$ is large enough. (Here, and in what follows, we write $\log$ for the natural logarithm, and $\log _a$ for the logarithm to base $a$.)

They also showed in \cite{AKMMR}, that the correlation measure $C_k(E_N)$ is concentrated around its mean $\E[C_k]$. Namely, for all $\varepsilon >0$ and integer function  $k=k(N)$ with $2\leq k\leq \log N -\log\log N$  the probability that
\begin{equation*}
 1-\varepsilon<\frac{C_k(E_N)}{\E[C_k]}<1+\varepsilon
\end{equation*}
holds is at least $1-\varepsilon$ if $N$ is large enough.

Recently, K.-U. Schmidt studied the limiting distribution of  $C_k(E_N)$ \cite{kai}. He showed that if $e_1,e_2,\ldots \in\{-1,+1\}$ are chosen independently and uniformly, then for fixed $k$
\begin{equation*}
 \frac{C_k(E_N)}{\sqrt{2N\log \binom{N}{k-1}}}\rightarrow 1 \quad \text{almost surely},
\end{equation*}
as $N\rightarrow \infty$, where $E_N=(e_1,\dots, e_N)$.

Let us now turn to the minimal value of $C_k(E_N)$. Clearly,
\begin{equation*}
 \min \{C_{k}(E_N): \ E_N\in\{-1,+1\}\}=1 \text{ for odd } k,
\end{equation*}
where the minimum is reached by the alternating sequence $(1,-1,1,-1,\dots)$. However, for even order, Alon,  Kohayakawa, Mauduit,  Moreira and R\"odl \cite{AKMMR-min} showed that
\begin{equation}\label{eq:c_k-min}
 \min \{C_{2k}(E_N): \ E_N\in\{-1,+1\}\}>\sqrt{\frac{1}{2} \left\lfloor \frac{N}{2k+1}\right\rfloor},
\end{equation}
see also \cite{kai}.

In order to study the pseudorandomness of \textit{families} of finite binary sequences instead of single sequences, Gyarmati, Mauduit and S\'ark\"ozy \cite{cross-correlation} introduced the notion of the \textit{cross-correlation measure} (see also the survey paper \cite{sarkozy-family-survey}).

\begin{definition}
For positive integers $N$ and $S$, consider a map
\[
 G_{N,S}:\{1,2,\dots, S\} \rightarrow \{-1,+1\}^N,
\]
and write $G_{N,S}(s)=(e_1(s),\dots, e_N(s))\in\{-1,1\}^N$ ($1\leq s \leq S$).

The \textit{cross-correlation measure $\Phi_k \left(G_{N,S}\right)$ of order $k$} of  $G_{N,S}$ is defined as
\[
 \Phi_k \left(G_{N,S}\right)=\max\left|\sum_{n=1}^M e_{n+d_1}(s_1)\cdots e_{n+d_{k}}(s_k)\right|,
\]
where the maximum is taken over all integers $M, d_1,\dots, d_k$ and $1\leq s_1,\dots, s_k\leq S$
such that $0\leq d_1\leq d_2\leq \dots \leq d_k<M+d_k\leq N$ and $d_i\neq d_j$ if $s_i=s_j$.
\end{definition}

We remark that in \cite{cross-correlation} only injective maps $G_{N,S}$ were considered, and the cross-cor\-re\-la\-tion measure is defined for the families $\F=\{G_{N,S}(s): \ s=1,2,\dots, S\}$ of size $S$.

The typical order of magnitude of $\Phi_k \left(G_{N,S}\right)$ was established  in \cite{merai-cross-correlation-typical} for large range of $k$ and  for random maps $G_{N,S}$, i.e. when all $e_n(s)\in\{-1,+1\}$  ($1\leq n \leq N$, $1\leq s \leq S$) are chosen independently and uniformly.

\begin{thm}\label{thm:typical}
For a given $\varepsilon >0$, there exists $N_0$, such that if $N>N_0$ and $1\leq \log_2 S< N/12$, 
then we have with probability at least $1-\varepsilon$, that
\begin{align*}
 \frac{2}{5} \sqrt{N\left(\log \binom{N}{k}+k\log S\right)} <\Phi_k \left(G_{N,S}\right) &< \frac{5}{2}\sqrt{N\left(\log \binom{N}{k}+k\log S \right)}
\end{align*}
for every integer $k$ with $2\leq k \leq N/(6\log_2 S)$.
\end{thm}

Our first result tells that analogously to the correlation measure of binary sequences, the cross-correlation measure of families $\Phi_k \left(G_{N,S}\right)$ is concentrated around its mean $\E\left[ \Phi_k \left(G_{N,S}\right) \right]$ if $k$ is small enough.

\begin{thm}\label{thm:conv-prob}
 For any fixed constant $\varepsilon>0$ and any integer function $k=k(N)$ with $2\leq k \leq (\log N + \log S)/ \log \log N$, there is a constant $N_0\geq 12 \log_2 S$ for which the following holds. If $N\geq N_0$, then the probability that
 \[
  1-\varepsilon < \frac{\Phi_k(G_{N,S})}{\E\left[\Phi_k(G_{N,S}) \right]}<1+\varepsilon
 \]
holds is at least $1-\varepsilon$.
\end{thm}

Next, we improve the upper bound in Theorem \ref{thm:typical}.

\begin{thm}\label{thm:upperbound}
For  positive integers $N,S$ and for $1\leq s\leq S$ write $G_{N,S}(s)=(e_1(s),e_2(s),\dots,\allowbreak e_N(s))$ ($1\leq s\leq S$). Let $e_n(s)$  ($1\leq n \leq N$, $1\leq s\leq S$) be drawn independently and uniformly at random from $\{-1,1\}$. For all $\varepsilon>0$ we have
\begin{multline*}
  \P\Bigg[ \Phi_k\left(G_{N,S}\right)\leq (1+\varepsilon)\sqrt{2N \log\left( \binom{SN}{k}-\binom{S(N-1)}{k} \right)} \\
  \text{for all $k$ satisfying } 2\leq k < SN  \Bigg]\rightarrow 1,
\end{multline*}
as $N\rightarrow \infty$.
\end{thm}

In order to obtain the asymptotic distribution of the cross-correlation measure $\Phi_k \left(G_{N,S}\right)$, 
consider the set $\Omega$ of all maps $G_S:\N \rightarrow \{-1,+1\}^{\N\times S}$ and write $G_S(s)=(e_1(s),e_2(s),\dots)$ for $1\leq s \leq S$. Let us endow $\Omega$  the  probability measure
\begin{equation}\label{eq:prob}
 \P[G_S \in\Omega: 
 \ e_1(i)=c_{1,i}, e_2(i)=c_{2,i},\dots, e_N(i)=c_{N,i}, \ i=1,\dots, S ]=2^{-NS}
\end{equation}
for all $N\in\N$ and all $(c_{1,i},c_{2,i},\dots, c_{N,i})_{i=1}^S\in\{-1,1\}^{N\times S}$.

\begin{thm}\label{thm:limit}
Let $S$ be a positive integer and let $G_S$ be drawn from $\Omega$ equipped with the  probability measure defined by \eqref{eq:prob} with $G_S(s)=(e_1(s),e_2(s),\dots)$ for $1\leq s \leq S$. Let $G_{N,S}(s)=(e_1(s),\dots,e_N(s))$ for $1\leq s \leq S$. For a fixed $k\geq 2$ we have 
 \[
  \frac{\Phi_k\left(G_{N,S}\right) }{\sqrt{2N\log \binom{N}{k-1} }} \rightarrow 1 \quad \text{almost surely}
 \]
as $N\rightarrow \infty$.
\end{thm}

Finally, we study the minimum values of the cross-correlation measure $\Phi_k \left(G_{N,S}\right)$. If $G_{N,S}$ is non-injective, say $G_{N,S}(S-1)=G_{N,S}(S)$, then
\[
 \Phi_{k} \left(G_{N,S}\right)=\max\left\{\Phi_{k-2} \left(G_{N,S-1}\right), \Phi_{k} \left(G_{N,S-1}\right)\right\},
\]
thus it is enough to control the minimum values of $\Phi_{k}\left(G_{N,S}\right)$ when $G_{N,S}$ is injective.

In \cite{cross-correlation}, Gyarmati, Mauduit and S\'ark\"ozy showed, that if the order of the measure is odd and $S$ is small, then $\Phi_{2k+1} \left(G_{N,S}\right)$ can be small.
\begin{prop}
Let $N\in\N$, $k,S\in\N$, such that $2k+1<N$, $S<N$. Then there is an injective map $G_{N,S}$ such that
\[
  \Phi_{2k+1} \left(G_{N,S}\right)\leq 2S.
\]
\end{prop}
Based on this observation they posed the following problem.
\begin{problem}
Estimate $\min \left\{\Phi_{2k+1} \left(G_{N,S}\right)\right\}$ for any fixed $N,k$ and $S$, where the minimum is taken over all injective maps $G_{N,S}:\{1,2,\dots, S\}\rightarrow \{-1,1\}^N$.
\end{problem}

We shall prove
\begin{thm}\label{thm:min-odd}
 If $k$ and $N$ are positive integers, then
 \[
 \lfloor\log_2 S-\log_2 (2k+1) \rfloor \leq \min \left\{\Phi_{2k+1}\left(G_{N,S}\right)\right\}  \leq  \lceil \log_2 S \rceil,
 \]
where the minimum is taken over all injective maps $G_{N,S}:\{1,2,\dots, S\}\rightarrow \{-1,1\}^N$.
\end{thm}

Similarly to the correlation measure, the cross-correlation measure cannot be small if its order is even.
From \eqref{eq:c_k-min} and a trivial estimate we get
\[
\Phi_{2k} \left(G_{N,S}\right) \geq \max \left\{ C_{2k}\left(G_{N,S}(s)\right) : 1\leq s\leq S\right\}\geq \sqrt{\frac{1}{2}\left\lfloor \frac{N}{2k+1}\right\rfloor  }.
\]
This lower bound can be improved, for example, by essentially a $\log \lfloor S/k\rfloor $ term if $S$ is large.

\begin{thm}\label{thm:min-even}
If $k$ and $N$ are positive integers, then for all injective maps $G_{N,S}:\{1,2,\dots, S\}\rightarrow \{-1,1\}^N$ we have
\begin{equation}\label{eq:min-1}
   \Phi_{2k} \left(G_{N,S}\right) \geq \sqrt{ \frac{1}{50}{N\log \lfloor S/k \rfloor} \left/ { \log \frac{50 N}{\log \lfloor S/k \rfloor}  } \right. }
\end{equation}
if $ 2kN\leq S$, and
\begin{equation}\label{eq:min-2}
   \Phi_{2k} \left(G_{N,S}\right) \geq \sqrt{\frac{N}{2\lceil k/S \rceil+1}}
\end{equation}
if  $ 2kN>S$.
\end{thm}

\section{Estimates for $\Phi_{k} \left(G_{N,S}\right)$ for random $G_{N,S}$}

It this section we shall prove Theorems \ref{thm:conv-prob} and \ref{thm:upperbound}.
The proof of Theorem \ref{thm:conv-prob} is based on the following result (see e.g. \cite[Lemma~1.2]{McDiarmid}).

\begin{lemma}\label{lemma:McDiarmid}
 Let $X_1,\dots,X_n$ be independent random variables, with $X_j$ taking values in a set $A_j$ for each $j$. Suppose that the (measurable) function $f: \prod_{j=1}^nA_j\rightarrow \mathbb{R}$ satisfies
 \begin{equation*}
  |f(\mathbf{x})-f(\mathbf{x}')|\leq c_j 
 \end{equation*}
whenever the vectors $\mathbf{x}$ and $\mathbf{x}'$ differ only in the $j$th co-ordinate. Let $Y$ be the random variable $f(X_1,\dots, X_n)$. Then for any $\theta>0$,
\begin{equation*}
 \P\left[|Y-\E(Y)|\geq \theta \right]\leq 2 \exp\left\{-\frac{2\theta^2}{\sum_{j=1}^nc_j^2}\right\}.
\end{equation*}
\end{lemma}

\begin{lemma}\label{lemma:distance}
For $\theta \geq 0$ we have
\begin{equation}\label{eq:distance} 
\P\left[\left|\Phi_k \left(G_{N,S}\right)-\E\left[\Phi_k\left(G_{N,S}\right)\right]\right|\geq \theta\right]\leq 2\exp \left\{-\frac{\theta ^2}{2k^2N}\right\}.
\end{equation}
\end{lemma}

\begin{proof}
For a fixed $1\leq j\leq N$ consider two maps $G_{N,S}, G'_{N,S}: \{1,\dots, S\}\rightarrow \{-1,1\}^N$  with $G_{N,S}(s)=(e_1(s),\dots, e_N(s))$ and $G'_{N,S}(s)=(e'_1(s),\dots, e'_N(s))$ for $1\leq s \leq S$, such that 
for all $ s$ the sequences $(e_1(s),\dots, e_N(s))$ and  $(e'_1(s),\dots, \allowbreak e'_N(s))$ can only differ at the $j$th position:
\[
 e_n(s)=e'_n(s) \text{ for $1\leq s\leq S$ and $n\neq j$}.
\]
Then
\[
 \left|\left|\sum_{n=1}^M e_{n+d_1}(s_1)\dots e_{n+d_{k}}(s_k)\right|-\left|\sum_{n=1}^M e'_{n+d_1}(s_1)\dots e'_{n+d_{k}}(s_k)\right|\right|\leq 2k,
\]
therefore
\[
 \left|\Phi_k\left(G_{N,S}\right)-\Phi_k\left(G'_{N,S}\right)\right|\leq 2k
\]
which proves \eqref{eq:distance} by Lemma \ref{lemma:McDiarmid}.
\end{proof}

\begin{proof}[Theorem \ref{thm:conv-prob}]
By Lemma \ref{lemma:distance} it is enough to show that if $2\leq k\leq (\log N + \log S)/\log \log N$, then taking $\theta=\varepsilon\, \E [\Phi_k(G_{N,S})]$ the right hand side of \eqref{eq:distance} is $o(1)$. If $N$ is large enough, then by Theorem \ref{thm:typical} we have
 \[
  \E[\Phi_k(G_{N,S})] >\frac{1}{5} \sqrt{kN\left( \log N + \log S\right)}.
 \]
Then
\[
 \frac{\theta^2}{2k^2 N}=\frac{\varepsilon^2 (\E[\Phi_k(G_{N,S})])^2 }{2k^2N}\geq\frac{\varepsilon^2}{50} \frac{\log N+\log S}{k}\rightarrow \infty,
\]
as $N\rightarrow \infty$. 
\end{proof}

Let $X_1,\dots , X_N$ be independent random variables, each taking the values -1 or 1, each with probability $1/2$. Define the random variable
 \[
  R_N=\max_{1\leq m_1\leq m_2\leq N}\left|\sum_{j=m_1}^{m_2}X_j \right|.  
\]

The following lemma states an estimate for large deviation of $R_N$ \cite[Lemma~2.2]{kai}.

\begin{lemma}\label{lemma:randomwalk}
For all $\delta >0$, there exists $N_0=N_0(\delta)$ such that for all $N\geq N_0$ and all $\lambda>2\sqrt{N}$ we have
\[
 \P\left[R_N>(1+\delta)\lambda \right]\leq \log N \exp \left(-\frac{\lambda^2}{2N} \right).
\]
\end{lemma}

One can obtain in the same way as \cite[Claim 18]{AKMMR} that the summands in the definition of the cross-correlation measure are pairwise independent.

\begin{lemma}\label{lemma:independence}
 Let $1\leq n,n'\leq N$ be integers, $(s_1,\dots, s_k)$, $(s_1',\dots, s_k')$, $(d_1,\dots, d_k)$ and $(d_1',\dots, d_k')$ be $k$-tuples such that $1\leq s_1,\dots, s_k\leq S$, $1\leq s_1',\dots, s_k'\leq S$,  $0\leq d_1 \leq \dots \leq  d_k$, $0\leq d_1' \leq \dots \leq  d_k'$ with $d_i\neq d_j$ if $s_i=s_j$ and $d_i'\neq d_j'$ if $s_i'=s_j'$. If
 \[
  (n,s_1,\dots, s_k,d_1,\dots, d_k)\neq (n',s_1',\dots, s_k',d_1',\dots, d_k'),
 \]
then 
\[
 e_{n+d_1}(s_1)\cdots e_{n+d_k}(s_k) \quad \text{and} \quad  e_{n'+d_1'}(s_1')\cdots e_{n'+d_k'}(s_k')
\]
are independent.
\end{lemma}

Throughout the proofs of Theorems \ref{thm:upperbound} and \ref{thm:limit} we will frequently use the following well-known bounds to the binomial coefficients
\begin{equation}\label{eq:binom-bound}
 \left(\frac{n}{m}\right)^m\leq \binom{n}{m}\leq \left(\frac{en}{m}\right)^m, \quad \text{ for } n,m\in\mathbb{N}, \quad 0< m \leq n. 
\end{equation}

\begin{proof}[Theorem \ref{thm:upperbound}]
Write $G_{N,S}(s)=(e_1(s),\dots, e_N(s))$ for $1\leq s\leq S$. Then writing $d_1=0$ we have
\begin{equation}\label{eq:thm3-eq1}
  \Phi_k \left(G_{N,S}\right)=\max_{0\leq d_2\leq \dots \leq d_k} \ \sideset{}{^*}\max_{s_1,\dots, s_k} \max_{1\leq m_1\leq m_2\leq N-d_k}\left|\sum_{n=m_1}^{m_2} e_{n+d_1}(s_1)\dots e_{n+d_{k}}(s_k)\right|,
\end{equation}
where the asterisk indicates that the  second maximum is taken over all $1\leq s_1,\dots, s_k\leq S$ such that $s_i\neq s_j$ if $d_i=d_j$.

Let
\[
 \lambda=\sqrt{2N \log \left( \binom{SN}{k}-\binom{S(N-1)}{k} \right) }
\]
and write $1+\varepsilon=\sqrt{1+\gamma}(1+\delta)$ for some $\gamma, \delta >0$. By Lemmas \ref{lemma:randomwalk} and \ref{lemma:independence} we have
\begin{equation*}
 \P\left[\max_{1\leq m_1\leq m_2\leq N-d_k}\left|\sum_{n=m_1}^{m_2} e_{n+d_1}(s_1)\dots e_{n+d_{k}}(s_k)\right|>(1+\varepsilon)\lambda\right]
\end{equation*}
is at most
\[
 \log N \exp\left\{ -\frac{\lambda^2(1+\gamma)}{2N} \right\}=\frac{\log N}
 { \left( \binom{SN}{k}-\binom{S(N-1)}{k} \right) ^{(1+\gamma)} }
\]
if $N$ is large enough.

Summing over all tuples $(d_2,\dots, d_k)$ and $(s_1,\dots, s_k)$ considered in \eqref{eq:thm3-eq1} we get
\begin{equation}\label{eq:thm3-end1}
\P\left[ \Phi_k (G_N)>(1+\varepsilon)\lambda \right]\leq \sum_{0\leq d_2\leq \dots \leq d_k} \ \sideset{}{^*}\sum_{s_1,\dots, s_k}\frac{\log N}{\left( \binom{SN}{k}-\binom{S(N-1)}{k} \right) ^{(1+\gamma)}}
\end{equation}
if $N$ is large enough.

Denoting the number of zero $d_i$'s by $\ell$, we get that the number of possible tuples is
\begin{equation}\label{eq:thm3:CV}
  \sum_{\ell=1}^{S}\binom{S}{\ell}\binom{S(N-1)}{k-\ell}=\binom{SN}{k}-\binom{S(N-1)}{k},
\end{equation}
where the equation follows from the Chu--Vandermonde identity
\[
 \sum_{j=0}^{l}\binom{m}{j}\binom{n-m}{l-j}=\binom{n}{l}
\]
which can be obtained from the coefficient of $x^l$ in the polynomial equation $(1+x)^{m}(1+x)^{m-n}=(1+x)^{n}$. 

From \eqref{eq:thm3-end1} and \eqref{eq:thm3:CV} we get
\begin{align*}
 \P\left[ \Phi_k (G_N)>(1+\varepsilon)\lambda \right]\leq \frac{\log N}
 { \left( \binom{SN}{k}-\binom{S(N-1)}{k} \right) ^{\gamma} }
 &\leq \frac{\log N} { \left( \binom{SN}{k}-\binom{SN-1}{k} \right) ^{\gamma} } \notag \\
 &= \frac{\log N} {  \binom{SN-1}{k-1} ^{\gamma} }.
\end{align*}

In order to prove the theorem it is enough to show that
\[
 \sum_{k=2}^{SN-1}\P\left[ \Phi_k (G_N)>(1+\varepsilon)\lambda \right] \rightarrow 0, \quad \text{as } N\rightarrow \infty.
\]

Let $M$ be an integer such that $M\gamma>1$. Then, for $N>M/S$ we have that
\begin{align*}
 \sum_{k=2}^{SN-1} \P\left[ \Phi_k (G_N)>(1+\varepsilon)\lambda \right]
 &\leq 2\sum_{k=1}^{M-1}\frac{\log N} {  \binom{SN-1}{k} ^{\gamma}}+2\sum_{k=M}^{\lfloor (SN-1)/2 \rfloor}\frac{\log N} {  \binom{SN-1}{k} ^{\gamma}}\notag\\
 &\leq \frac{2M \log N}{(SN-1)^{\gamma}}+\frac{SN\log N}{\binom{SN-1}{M}^\gamma}\notag\\ 
 &\leq \frac{2M \log N}{(SN-1)^{\gamma}}+\frac{M^{M\gamma} \log N}{(SN-1)^{M\gamma-1}},
\end{align*}
using \eqref{eq:binom-bound}.
Since $\gamma>0$ and $M\gamma >1$, the right hand side tends to zero as $N\rightarrow \infty$ which proves the theorem. 
\end{proof}

\section{Limiting distribution}

The main ingredient of the proof of Theorem \ref{thm:limit} is the following asymptotic result for the mean $\E\left[\Phi_k\left(G_{N,S}\right) \right]$.

\begin{lemma}\label{lemma:conv-E}
Let $G_{N,S}(s)$ be drawn independently and uniformly at random from $\{-1,1\}^N$ for all $1\leq s \leq S$. Then, as $N\rightarrow \infty$,
 \[
  \frac{\E\left[\Phi_k\left(G_{N,S}\right) \right]}{\sqrt{2N(k-1) \log N }}\rightarrow 1.
 \]
\end{lemma}

Let $\ell$ and $k_1,\dots, k_{\ell}$ be positive integers with $k_1+\dots +k_\ell=k$. Let $D=(d_1^i,\dots, d_{k_i}^i)_{i=1}^\ell$ be a $k$-tuple such that
\begin{equation}\label{eq:D-form}
 0\leq d_1^i<\dots<d_{k_i}^i\leq \frac{N}{\log N} \quad \text{for } i=1,\dots, \ell, \quad \text{and } \min_{i=1,\dots, \ell}d_1^i=0.
\end{equation}
For distinct $1\leq s_1,\dots, s_\ell\leq S$ write
\begin{align*}
&V_{k_1,\dots, k_\ell}(G_N,s_1,\dots, s_\ell,D)\\
&= \sum_{n=1}^{N-\lfloor \frac{N}{\log N} \rfloor} e_{n+d_1^1}(s_1)\dots e_{n+d_{k_1}^1}(s_1) \dots e_{n+d_1^\ell}(s_\ell)\dots e_{n+d_{k_\ell}^\ell}(s_\ell).
\end{align*}

For functions $f(x),g(x)$, we use the standard notation $f(x)\sim g(x)$ to mean $f(x)=g(x)(1+o(1))$ as $x\rightarrow \infty$. 

\begin{lemma}\label{lemma:lowerbound1}
Let $G_{N,S}(s)$ be drawn independently and uniformly at random from $\{-1,1\}^N$ for all $1\leq s \leq S$. Then 
\begin{multline*}
 \P\left[ \left|V_{k_1,\dots, k_\ell}\left(G_{N,S},s_1,\dots, s_\ell,D\right)\right|   \geq  \sqrt{2N(k-1)\log N }\right]\\
  \sim\frac{1}{e^{k-1}N^{k-1}  \sqrt{\pi  (k-1)\log N}}
\end{multline*}
as $N\rightarrow \infty$. 
\end{lemma}

We need the following form of the de Moivre-Laplace theorem (see, e.g., \cite[Chapter I, Theorem 6]{Bollobas}).

\begin{lemma}\label{lemma:tail}
Let $X_1,\dots , X_n$ be independent random variables, each taking the values -1 or 1, both with probability $1/2$.
For any $c_n>0$ with $c_n=o(n^{1/6})$ and $c_n \rightarrow \infty$, we have 
\begin{equation*}
\P\left[\left| \sum_{i=1}^n X_i \right| \geq c_n \sqrt{n} \right] \sim \sqrt{\frac{2}{\pi}} \frac{1}{c_n} \exp\left\{-\frac{c_n^2}{2}\right\}.
\end{equation*}
\end{lemma}

\begin{proof}[Lemma \ref{lemma:lowerbound1}]
Write
\[
 c_N=\sqrt{\frac{2N}{N-\lfloor N/\log N \rfloor}(k-1)\log N  }.
\]
Then, by Lemmas \ref{lemma:independence} and \ref{lemma:tail} we have
\begin{align*}
&\P\left[ \left|\sum_{n=1}^{N-\lfloor N/ \log N \rfloor} e_{n+d_1}(s_1)\cdots e_{n+d_{k}}(s_k)\right|   \geq c_N\sqrt{N-\lfloor N/\log N \rfloor}\right] \notag \\
&  \sim \frac{1}{\sqrt{\pi  (k-1)\log N}}\exp\left\{ -\frac{N}{N- \frac{N}{\log N} +O(1) }(k-1)\log N\right\}
\notag \\
&  = \frac{1}{e^{k-1}N^{k-1}  \sqrt{\pi  (k-1)\log N}}e^{-\frac{k-1}{\log N-1}-O\left(k\frac{\log N }{N}\right)}
\end{align*}
if $N$ is large enough. 
\end{proof}

\begin{lemma}\label{lemma:intesection}
Let $G_S$ be drawn from $\Omega$ with the probability measure defined by \eqref{eq:prob} and define $G_{N,S}$ as in Theorem \ref{thm:limit}.

Let  $\ell$ and $k_1,\dots, k_\ell, k'_1,\dots, k'_{\ell}$ be positive integers with $k_1+\dots +k_\ell=k'_1+\dots +k'_{\ell}=k$, let $1\leq s_1< \cdots< s_\ell\leq S$ and  $D \neq D'$ $k$-tuples having the form \eqref{eq:D-form}. Then writing
\[
 \lambda =\sqrt{2N(k-1)\log N }
\]
we have
\begin{align*}
 &\P[| V_{k_1,\dots, k_\ell}(G_N,s_1,\dots, s_\ell,D) |\geq \lambda \cap | V_{k'_1,\dots, k'_{\ell }}(G_N,s_1,\dots, s_{\ell },D') |\geq \lambda  ]\\
 &\leq \frac{23}{N^{2(k-1)}}.
\end{align*}
\end{lemma}

In order to prove Lemma \ref{lemma:intesection} we use the following notation. A tuple $(x_1,\dots, \allowbreak x_{2m})$ is \emph{$t$-even} if there exists a permutation $\sigma$ of $\{1,2,\dots, 2m\}$ such that $x_{\sigma(2i-1)}=x_{\sigma(2i)}$ for each $i\in\{1,\dots,t\}$ and $t$ is the largest integer with this property. An $m$-even tuple is just called \textit{even}.

The following lemma gives an upper bound to the number of even tuples \cite[Lemma 2.4]{kai_k=2}.

\begin{lemma}
Let $m$ and $q$ be positive integers. Then the number of even tuples in $\{1,\dots, m\}^{2q}$ is at most $(2q-1)!!m^q$, where the $(2q-1)!!$ semi-factorial is defined as
\[
  (2q-1)!!=\frac{(2q)!}{q!2^q}=(2q-1)\cdot(2q-3)\cdots 3\cdot 1.
\]
\end{lemma}

The following result is an extension of \cite[Lemma 3.7]{kai}.

\begin{lemma}\label{lemma:d-even}
Let $N$, $q$ and  $t$ be positive integers satisfying $0\leq t <q$. Let $D$ and $D'$ be two $k$-tuples satisfying $D\neq D'$ and \eqref{eq:D-form}.

If $(x_1,\dots, x_{2q})$ is $d$-even for some $d<q-t$, then the number of $4q$-tuples $(x_1,\dots,x_{2q},\allowbreak y_1,\dots, y_{2q})$ in $\{1,\dots,N\}^{4q}$ such that for each $1\leq s\leq S$ the tuple
\[
 (x_i+d_1^s,\dots, x_i+d_{k_s}^s,y_i+d_1'^s,\dots, y_i+d_{k_s'}'^s)_{i=1}^{2q}
\]
is even, is at most 
\[
 (4kq-1)!!N^{2q-(t+1)/3}.
\]
\end{lemma}

\begin{proof}
Since the proof is similar to the proof of \cite[Lemma 3.7]{kai}, we leave some details to the reader.

We construct a set of tuples that contains the required $4q$-tuples as a subset. For each $1\leq s\leq S$, arrange the $4(k_s+k_s')q$ variables
\begin{equation}\label{eq:arrange}
 x_i+d_j^s,\ y_i+d_l'^s  \quad \text{for } 1\leq i\leq 2q, \ 1\leq j \leq k_s, \ 1\leq l \leq k_s'
\end{equation}
into $2(k_s+k_s')q$ unordered pairs $\{a_1^s,b_1^s\},\dots, \{a_{2(k_s+k_s')q}^s,b_{2(k_s+k_s')q}^s\}$ such that there are at most $k(q-t-1)$ pairs of form $\{x_i+d_j^s,x_{i'} +d_j^s\}$. This can be done in at most $(4kq-1)!!$ ways. We formally set $a_i^s=b_i^s$ for all $1\leq i\leq 2q$ and all $1\leq s\leq S$. If this assignment does not yield a contradiction, then we call the arrangement \eqref{eq:arrange} \emph{consistent}.

If there is a pair of form $\{x_i+d_j^s, x_{i'}+d_l^s\}$ with $j\neq l$ in a consistent arrangement, then $i\neq i'$ and $x_i$ determines $x_{i'}$. Likewise, if there is a pair of form $\{y_i+d_j'^s, y_{i'}+d_l'^s\}$ with $j\neq l$ in a consistent arrangement, then $i\neq i'$ and $y_i$ determines $y_{i'}$. On the other hand, if a consistent arrangement consists a pair of form $\{x_i+d_j^s,y_{i'}+d_l'^s\}$, then $x_i$ determines $y_{i'}$ and at least one other variable in the list 
\begin{equation}\label{eq:var}
x_1,\dots, x_{2q},y_1,\dots, y_{2q}.
\end{equation}

Indeed, a consistent arrangement cannot contain pairs involving only the variables
\[
 x_i+d_1^s,\dots, x_i+d_{k_s}^s, y_{i'}+d_1'^s,\dots, y_{i'}+d_{k'_s}'^s \quad \text{for each } 1\leq s\leq S.
\]
Since for all $1\leq s\leq S$, $0\leq d_1^s<\dots <d_{k_s}^s$, and $0\leq d_1'^s<\dots <d_{k'_s}'^s$, the only possibility for such pairs would be 
\begin{equation}\label{eq:arrangement-2}
\{x_i+d_1^s,y_{i'}+d_1'^s\},\dots, \{x_i+d_{k_s}^s,y_{i'}+d_{k'_s}'^s\}, \quad 1\leq s \leq S.
\end{equation}
Let $u, v$ be two indices such that $d_1^u=0$ and $d_1'^v=0$. Then $x_i=y_{i'}+d_1'^u$ and $x_i+d_1^v=y_{i'}$ so we have $d_1^u=d_1'^v=0$, thus $x_i=y_{i'}$ and $d_j^u=d_j'^u$ ($j=1,\dots, k_u$), $d_j^v=d_j'^v$ ($j=1,\dots, k_v$). Moreover it also follows from \eqref{eq:arrangement-2} that $d_j^s=d_j'^s$ ($j=1,\dots, k_s$) for $s\neq u,v$, so $D=D'$, a contradiction.  

Now, by assumption, each consistent arrangement contains at most $k(q-t-1)$ pairs of form $\{x_i+d_j^s,x_{i'} +d_j^s\}$ and at most $kq$ pairs of the form  $\{y_i+d_j'^s,y_{i'} +d_j'^s\}$, and so at most
\[
 q-t-1+q+\frac{1}{3}(2t+2)=2q-\frac{1}{3}(t+1)
\]
of the variables in \eqref{eq:var} can be chosen independently. We assign to each of these a value of $\{1,\dots, N\}$. In this way, we construct a set of at most $(4kq-1)!!N^{2q-(t+1)/3}$ tuples that contains the required $4q$-tuples.
\end{proof}

\begin{lemma}\label{lemma:intesection-E} 
Let $S$ and $k$ be integers. Let $G_{N,S}(s)$ be drawn independently and uniformly at random from $\{-1,1\}^N$ for all $1\leq s \leq S$. 
Let  $\ell$ and $k_1,\dots, k_\ell$, $k'_1,\dots, k'_{\ell}$ be positive integers with $k_1+\dots +k_\ell=k'_1+\dots +k'_{\ell}=k$, let   $s_1,\dots, s_\ell\in\S$ distinct elements and let $D \neq D'$ $k$-tuples having the form \eqref{eq:D-form}.

For $0\leq h<p$ we have

\begin{multline}\label{eq:intesection}
 \E \left[  \left(V_{k_1,\dots, k_\ell}(G_N,s_1,\dots, s_\ell,D)  V_{k'_1,\dots, k'_{\ell }}(G_N,s_1,\dots, s_{\ell },D')\right)^{2p}\right]\\
 \leq N^{2p} ((2p-1)!!)^2 \left(1+\frac{(4kp)^{4kh}}{N^{1/3}}+\frac{(4kp)^{2kp}}{N^{(h+1)/3}}  \right).
\end{multline}
\end{lemma}

\begin{proof}[Lemma \ref{lemma:intesection-E}]
Since the proof is similar to the proof of \cite[Lemma 3.8]{kai}, we leave some details to the reader.

Expanding the left hand side of \eqref{eq:intesection}, we get that the expected value in \eqref{eq:intesection} is
\begin{multline}\label{eq:intesection-excepted}
 \sum_{n_1,\dots, n_{2p}=1}^{N-\lfloor \frac{N}{\log N} \rfloor} \sum_{m_1,\dots, m_{2p}=1}^{N-\lfloor \frac{N}{\log N} \rfloor} \E\Bigg[
\prod_{z=1}^{2p}
\prod_{u=1}^{\ell}   e_{n_z+d^u_1}(s_u)\dots e_{n_z+d^u_{k_u}}(s_u)\\
\cdot\prod_{v=1}^{\ell}   e_{m_z+d'^v_1}(s_v)\dots e_{m_z+d'^v_{k'_v}}(s_v)
 \Bigg].
\end{multline}

Since $e_n(s)$ are mutually independent for $1\leq n \leq N$ and $1\leq s \leq S$ with $e_n(s)\in\{-1,1\}$ and $\E[s_n(s)]=0$, then \eqref{eq:intesection-excepted} is the number of $4p$-tuples $(n_1,\dots, n_{2p}, m_1,\dots, m_{2p})$ such that for each $u$ the tuples $(n_z+d_1^u, \dots , n_z +d_{k_u}^u, m_z+d'^u_1, \dots , m_z +d'^u_{k'_u})_{z=1}^{2p}$ are even.

Then in the same way as in \cite[Lemma 3.8]{kai} one may get that the number of such $4p$-tuples is at most
\[
 \left((2p-1)!!N^{p}\right)^2\left(1+\frac{2p(k_i+k'_i)^{2h(k_i+k'_i)}}{N^{1/3}}+\frac{2p(k_i+k'_i)^{p(k_i+k'_i)}}{N^{(h+1)/3}} \right).
\]
using Lemma \ref{lemma:d-even}.
\end{proof}

\begin{proof}[Lemma \ref{lemma:intesection}]

If $X_1$ and $X_2$ are random variables, then for all positive integers $p$ and for $\theta_1,\theta_2>0$ Markov's inequality  yields
\begin{equation*}
 \P[|X_1|\geq \theta_1 \cap |X_2|\geq \theta_2]\leq \frac{\E[(X_1X_2)^{2p}]}{(\theta_1 \theta_2)^p}.
\end{equation*}

Let $p=\lfloor (k-1)\log N  \rfloor$ and $h=\lfloor \alpha  \log \log N \rfloor$ for some large $\alpha>0$ to be fixed later.

By \eqref{eq:intesection} and Markov's inequality we have
\begin{align*}
 &\P[| V_{k_1,\dots, k_\ell}(G_N,s_1,\dots, s_\ell,D) |\geq \lambda \cap | V_{k'_1,\dots, k'_{\ell }}(G_N,s_1,\dots, s_{\ell },D') |\geq \lambda  ]\\
 & \leq \frac{ \E \left[  \left(V_{k_1,\dots, k_\ell}(G_N,s_1,\dots, s_\ell,D)  V_{k'_1,\dots, k'_{\ell }}(G_N,s_1,\dots, s_{\ell },D')\right)^{2p}\right]}{\lambda^{4p}}\\
 &\leq \frac{\left((2p-1)!!\right)^2N^{2p}}{(2N(k-1)\log N )^{2p}}\left(1+\frac{(4kp)^{4kh}}{N^{1/3}}+\frac{(4kp)^{2kp}}{N^{(h+1)/3}}  \right)\\
 &= \frac{\left((2p-1)!!\right)^2}{(2(k-1)\log N )^{2p}}\left(1+K_1(p,h)+K_2(p,h)  \right),
\end{align*}
where
\begin{align*}
  K_1(p,h)=\frac{(4kp)^{4kh}}{N^{1/3}}, \quad
  K_2(p,h)=\frac{(4kp)^{2kp}}{N^{(h+1)/3}}.  
\end{align*}

By Stirling's approximation (see e.g. \cite{feller}) it follows that
\[
 \sqrt{2\pi n} \, n^n e^{-n}\leq n! \leq \sqrt{3\pi n}\, n^n e^{-n}
\]
so we have  
\[
 \frac{\left((2p-1)!!\right)^2}{(2(k-1)\log N )^{2p}}\leq \frac{3e^2}{N^{2(k-1)}}.
\]
Moreover
\begin{align*}
 \log K_1(p,h)&=-\frac{1}{3}\log N+4kh \log (4kp)
 \leq -\frac{1}{3}\log N +6\alpha k \log k (\log \log N )^2\\
 &\leq -\frac{1}{4}\log N
\end{align*}
if $N$ is large enough and
\begin{align*}
 \log K_2(p,h)&=-\frac{h+1}{3}\log N+2kp \log (4kp)\\
 &\leq -\frac{\alpha}{3}\log N \log \log N +6 k^2 \log k \log N \log \log N\\
 &=\left(-\frac{\alpha}{3} +6 k^2 \log k\right)\log N \log \log N\\
 &\leq -\log N \log \log N.
\end{align*}
if we choose $\alpha =10 k^2 \log k$. Then the result follows.
\end{proof}

\begin{proof}[Lemma \ref{lemma:conv-E}]
From Theorem \ref{thm:upperbound} and Lemma \ref{lemma:distance} it follows that
 \begin{equation}\label{eq:limsup}
 \limsup_{N\rightarrow \infty} \frac{\E\left[\Phi_k\left(G_{N,S}\right) \right]}{\sqrt{2N\log \left( \binom{S\cdot N}{k} -\binom{S\cdot(N-1)}{k} \right)}}\leq 1.
 \end{equation}
Now
\begin{align}\label{eq:binom-limit}
 &\binom{S\cdot N}{k} -\binom{S\cdot (N-1)}{k}\notag\\
 &=\binom{S\cdot N}{k}\left(1-\left(1-\frac{k}{S\cdot N}\right)\cdots \left(1-\frac{k}{S\cdot N-S+1}\right)  \right).
\end{align}
Since
\begin{multline*}
 \left(1-\frac{k}{S\cdot N}\right)\cdots \left(1-\frac{k}{S\cdot N-S+1}\right)\\
 \leq \left(1-\frac{k}{S\cdot N}\right)^{S}=1-\frac{k}{N}+O\left(\frac{k^2}{N^2}\right)
\end{multline*}
and
\begin{align*}
& \left(1-\frac{k}{S\cdot N}\right)\cdots \left(1-\frac{k}{S\cdot N-S+1}\right)\\
&\geq \left(1-\frac{k}{S(N-1)}\right)^{S}
 =1-\frac{k}{N-1}+O\left(\frac{k^2}{N^2}\right)
 =1-\frac{k}{N}+O\left(\frac{k}{N^2}\right),
\end{align*}
we have that \eqref{eq:binom-limit} is
\begin{equation}\label{eq:binom-final}
 \binom{S\cdot N}{k} -\binom{S\cdot (N-1)}{k}=\binom{S\cdot N}{k}\left(\frac{k}{N}+O\left(\frac{k}{N^2}\right)\right)
\end{equation}
and by \eqref{eq:binom-bound} its logarithm is
\[
\log \binom{S\cdot N}{k}+\log \frac{k}{N}+\log \left(1+ O\left(\frac{1}{N}\right)\right)\sim (k-1)\log N
\]
as $S$ and $k$ are fixed. 

It follows from \eqref{eq:limsup} that
\[
  \limsup_{N\rightarrow \infty} \frac{\E\left[\Phi_k\left(G_{N,S}\right) \right]}{\sqrt{2N (k-1)\log N }}\leq 1
\]
so it is enough to show that
 \[
 \liminf_{N\rightarrow \infty} \frac{\E\left[\Phi_k\left(G_{N,S}\right) \right]}{\sqrt{2N (k-1)\log N }}\geq 1.
 \]

Let $\delta>0$ and put
\[
 N(\delta)=\left\{N\geq k: \frac{\E\left[\Phi_k\left(G_{N,S}\right)\right]}{\sqrt{2N (k-1)\log N }}<1-\delta\right\}.
\]
We shall show, that $N(\delta)$ is a finite set for all $\delta >0$ which proves the lemma according to \eqref{eq:limsup}.
 
Clearly,
\[
 \Phi_k\left(G_{N,S}\right) \geq \max_{\ell, k_1,\dots, k_\ell}\max_{s_1,\dots, s_\ell} \max_{ D}  | V_{k_1,\dots, k_\ell}(G_{N,S},s_1,\dots, s_\ell,D)   |,
\]
so
\begin{align}\label{eq:split}
&\P\left[\Phi_k\left(G_{N,S}\right)\geq \lambda \right]\geq 
\sum
\P[| V_{k_1,\dots, k_\ell}(G_{N,S},s_1,\dots, s_\ell,D) |\geq \lambda  ] \notag\\
&-\frac{1}{2} 
\sum
\P\big[| V_{k_1,\dots, k_\ell}(G_{N,S},s_1,\dots, s_\ell,D) |\geq \lambda \notag \\
& \qquad \qquad \qquad \cap | V_{k'_1,\dots, k'_{\ell '}}(G_{N,S},s'_1,\dots, s'_{\ell '},D') |\geq \lambda  \big],
\end{align}
where the first sum is taken over all positive integers $\ell$ and $k_1,\dots, k_\ell$ with $k_1+\dots +k_\ell=k$, all distinct $1\leq s_1,\dots, s_\ell\leq S$ and all $D$ having the form \eqref{eq:D-form}, while the second sum is taken over all positive integers $\ell, \ell '$ and $k_1,\dots, k_\ell, k'_1,\dots, k'_{\ell'}$ with $k_1+\dots +k_\ell=k'_1+\dots +k'_{\ell'}=k$, all  distinct $1\leq s_1,\dots, s_\ell\leq S$ and distinct $1\leq s'_1,\dots, s'_{\ell'}\leq S$ and all $D,D'$ having the form \eqref{eq:D-form} with the additional restriction that $D\neq D'$ if $\{s_1,\dots, s_\ell\}=\{s'_1,\dots, s'_{\ell '}\}$.

First we give a lower bound to the first term of \eqref{eq:split} by Lemma \ref{lemma:lowerbound1}. Namely, 
\begin{align}\label{eq:lower-1}
 & \sum_{\ell =1}^k \sum_{\substack{k_1,\dots, k_\ell \geq 1 \\ k_1+\dots +k_\ell =k}}\sum_{1\leq s_1,\dots, s_\ell \leq S}\sum_{D} \P[| V_{k_1,\dots, k_\ell}(G_{N,S},s_1,\dots, s_\ell,D) |\geq \lambda  ] \notag \\
 \geq & \frac{1 }{2e^{k-1}N^{k-1} \sqrt{ \pi (k-1)\log N}} \sum_{\ell =1}^k \sum_{\substack{k_1,\dots, k_\ell \geq 1 \\ k_1+\dots +k_\ell =k}}\sum_{1\leq s_1,\dots, s_\ell \leq S}\sum_{D} 1\notag \\
 = &  \left(  \binom{S\cdot (\lfloor N/\log N \rfloor+1)}{k}- \binom{S\cdot \lfloor N/\log N \rfloor}{k}\right)  
 \frac{1 }{2e^{k-1}N^{k-1}  \sqrt{ \pi (k-1)\log N}} \notag \\
 \geq & \binom{S\cdot (\lfloor N/\log N \rfloor+1)}{k}\cdot \frac{k}{N} 
 \frac{1 }{4e^{k-1}N^{k-1}  \sqrt{ \pi (k-1)\log N}} \notag\\
 \geq & \frac{S^k }{4(ek)^{k-1} (\log N)^k \sqrt{ \pi (k-1)\log N}} 
\end{align}
using \eqref{eq:binom-bound} and \eqref{eq:binom-final} with $N$ replaced  by $\lfloor N/\log N\rfloor+1$.

We give a lower bound to the second term of \eqref{eq:split} by Lemmas \ref{lemma:lowerbound1} and  \ref{lemma:intesection}. If $\{s_1,\dots, s_\ell\}\neq \{s'_1,\dots, s'_{\ell'}\}$, then  $V_{k_1,\dots, k_\ell}(G_{N,S},s_1,\dots, s_\ell,D)$ and $V_{k'_1,\dots, k'_{\ell '}}(G_{N,S},\allowbreak s'_1,\dots, s'_{\ell '},D')$ are independent by Lemma \ref{lemma:independence}, thus by Lemma \ref{lemma:lowerbound1} we have in the same way that
\begin{align}\label{eq:split-2}
&  \sum_{\ell, \ell ' =1}^k \underset{  \{ s_1,\dots, s_\ell\}\neq \{s'_1,\dots, s'_{\ell'}\} }{\sum_{1\leq s_1,\dots, s_\ell \leq S} \sum_{1\leq s'_1,\dots, s'_\ell \leq S}}  \sum_{\substack{k_1,\dots, k_\ell \geq 1 \\ k_1+\dots +k_\ell =k}}\sum_{\substack{k'_1,\dots, k'_\ell \geq 1 \\ k'_1+\dots +k'_\ell =k}}\sum_{D,D'} \notag \\
&\quad \P[| V_{k_1,\dots, k_\ell}(G_{N,S},s_1,\dots, s_\ell,D) |\geq \lambda \cap | V_{k'_1,\dots, k'_{\ell '}}(G_{N,S},s'_1,\dots, s'_{\ell '},D') |\geq \lambda  ] \notag \\ 
&\leq \sum_{\ell, \ell ' =1}^k \underset{  \{s_1,\dots, s_\ell\}\neq \{s'_1,\dots, s'_{\ell'}\} }{\sum_{1\leq s_1,\dots, s_\ell \leq S} \sum_{1\leq s'_1,\dots, s'_\ell \leq S}}  \sum_{\substack{k_1,\dots, k_\ell \geq 1 \\ k_1+\dots +k_\ell =k}}\sum_{\substack{k'_1,\dots, k'_\ell \geq 1 \\ k'_1+\dots +k'_\ell =k}}\sum_{D,D'}\notag\\
&\qquad \qquad \qquad \qquad \cdot \frac{2 }{e^{2(k-1)}N^{2(k-1)}  \pi (k-1)\log N}
\notag \\
&=\left(  \binom{S\cdot (\lfloor N/\log N \rfloor+1)}{k}- \binom{S\cdot \lfloor N/\log N \rfloor}{k}\right)^2 \notag  \\
&\qquad \qquad \qquad \qquad \cdot \frac{2 }{e^{2(k-1)}N^{2(k-1)}  \pi (k-1)\log N} \notag \\
&\leq \left(\binom{S\cdot (\lfloor N/\log N \rfloor+1)}{k}\cdot \frac{k}{N}\right)^2  \frac{4 }{e^{2(k-1)}N^{2(k-1)}  \pi (k-1)\log N} \notag \\
&\leq \frac{4 S^{2k}}{k^{2(k-1)} (\log N)^{2k}\pi (k-1)\log N}
\end{align}
by \eqref{eq:binom-bound}.

For $\{s_1,\dots, s_\ell\}= \{s'_1,\dots, s'_{\ell'}\}$ we have by Lemma \ref{lemma:intesection}, that
\begin{align}\label{eq:split-1-b}
&  \sum_{\ell=1}^k \sum_{1\leq s_1,\dots, s_\ell \leq S}  \sum_{\substack{k_1,\dots, k_\ell \geq 1 \\ k_1+\dots +k_\ell =k}}\sum_{\substack{k'_1,\dots, k'_\ell \geq 1 \\ k'_1+\dots +k'_\ell =k}}\sum_{D,D'}\notag \\
&\quad \P[| V_{k_1,\dots, k_\ell}(G_{N,S},s_1,\dots, s_\ell,D) |\geq \lambda \cap | V_{k'_1,\dots, k'_{\ell '}}(G_{N,S},s'_1,\dots, s'_{\ell '},D') |\geq \lambda  ] \notag \\
&\leq  \sum_{\ell=1}^k \sum_{1\leq s_1,\dots, s_\ell \leq S}  \sum_{\substack{k_1,\dots, k_\ell \geq 1 \\ k_1+\dots +k_\ell =k}}\sum_{\substack{k'_1,\dots, k'_\ell \geq 1 \\ k'_1+\dots +k'_\ell =k}}\sum_{D,D'}
 \frac{23}{N^{2(k-1)}}
\end{align}

Since one of the $d_1'^{s_i}$ in $D'$ takes the value 0, for fixed $\ell$ and $s_1,\dots, s_\ell$, the number of possible $\ell$-tuples $(k_1',\dots, k_\ell')$ and $D'$ is at most
\[
 \ell \cdot \binom{\ell \cdot (\lfloor N /\log N\rfloor+1)}{k-1}\leq k\binom{k\cdot (\lfloor N /\log N\rfloor+1)}{k-1}\leq \frac{k^k 2^{k-1}N^{k-1}}{(k-1)^{k-1}(\log N)^{k-1}}
\]
by \eqref{eq:binom-bound}.
Thus \eqref{eq:split-1-b} is at most
\begin{align}\label{eq:split-1}
 &\frac{k^k 2^{k-1}N^{k-1}}{(k-1)^{k-1}(\log N)^{k-1}} \sum_{\ell=1}^k \sum_{1\leq s_1,\dots, s_\ell \leq S}  \sum_{\substack{k_1,\dots, k_\ell \geq 1 \\ k_1+\dots +k_\ell =k}}\sum_{D}
 \frac{23}{N^{2(k-1)}}\notag \\
 &\leq \frac{k^k 2^{k-1}N^{k-1}}{(k-1)^{k-1}(\log N)^{k-1}} \frac{23}{N^{2(k-1)}}\notag\\
 & \qquad \qquad \cdot \left(  \binom{S\cdot (\lfloor N/\log N \rfloor+1)}{k}- \binom{S\cdot \lfloor N/\log N \rfloor}{k}\right)^2    \notag \\
  &\leq \frac{23 k e^k 2^{2k-1}S^k}{(k-1)^{k-1}(\log N)^{2k-1}}  
\end{align}
by \eqref{eq:binom-bound}.

As $N\rightarrow \infty$, \eqref{eq:lower-1} dominates \eqref{eq:split-2} and \eqref{eq:split-1} thus we have from \eqref{eq:split} that
\begin{equation}\label{eq:lower-f}
 \P\left(\Phi_k\left(G_{N,S}\right)\geq \lambda \right)\geq \frac{S^k }{5(ek)^{k-1} (\log N)^k \sqrt{ \pi (k-1)\log N}}
\end{equation}

By the definition of $N(\delta)$, we have $\lambda > \E(\Phi_k (G_{N,S}))$ for $N\in N(\delta)$, thus by Lemma \ref{lemma:distance} we have for $\theta = \lambda- \E(\Phi_k (G_{N,S}))$ that
\[
 \P[\Phi_k (G_{N,S})\geq \lambda ]\leq 2\exp \left\{-\frac{(\lambda- \E(\Phi_k (G_{N,S})))^2}{2k^2N} \right\}
\]
for $N\in N(\delta)$. Comparing it with \eqref{eq:lower-f} we get
\[
 \frac{S^k }{5(ek)^{k-1} (\log N)^k \sqrt{ \pi ((k-1)\log N+k\log S)}}\leq 2\exp \left\{-\frac{(\lambda- \E(\Phi_k (G_{N,S})))^2}{2k^2N} \right\}
\]
i.e.
\begin{multline*}
 \frac{\E(\Phi_k (G_{N,S}))}{\sqrt{2N(k-1)\log N }}\\
 \geq 1-\sqrt{k^2 
 \frac{
 \log\left(10 (ek)^{k-1}\pi^{1/2}\right)+ k\log \log N +\frac{1}{2}\log(k-1)\log N-k\log S
  }
 {(k-1)\log N }}
\end{multline*}
Since the right hand side goes to 1 as $N \rightarrow \infty$, we see that the size of $N(\delta)$ is finite.
\end{proof}

Let
\[
 S^{\pm}(n)=\sum_{1\leq i\leq n} X_i,
\]
where $X_i$ ($1\leq i \leq n$) are independent random variables with mean 0, that is,
\[
 \P(X_i=-1)=\P(X_i=+1)=1/2.
\]

The following lemma states a well-known estimate for large deviation of $S^{\pm}(n)$ (see, e.g. \cite[Appendix~2]{AlonSpencer}):

\begin{lemma}\label{lemma:Spm}
 Let $X_i$ ($1\leq i \leq n$) be independent $\pm 1$ random variables with mean 0. Let  $S^{\pm}(n)=\sum_{1\leq i\leq n} X_i$. For any real number $a>0$, we have
 \[
  \P(S^{\pm}(n) >a)<e^{-a^2/2n}.
 \]

\end{lemma}

\begin{lemma}\label{lemma:speed}
Let $G_S$ be drawn from $\Omega$ equipped with the probability measure defined by \eqref{eq:prob} with $G_S(s)=(e_1(s),e_2(s),\dots)$ for $1\leq s\leq S$. Let $G_N(s)=(e_1(s),\dots,e_N(s))$ for $1\leq s\leq S$. Let $N_1,N_2,\dots$ be a strictly increasing sequence of integers. Then, almost surely
\[
 \Phi_k(G_{N_{r+1},S})-\Phi_k(G_{N_{r},S})\leq \sqrt{6 (N_{r+1}-N_r)(k-1)\log N_{r+1}}
\]
for all sufficiently large $r$.
\end{lemma}

\begin{proof}
Write
\[
  \lambda = \sqrt{6 (N_{r+1}-N_r)(k-1)\log N_{r+1}}.
\]
If
\begin{equation}\label{eq:distance-cel}
 \Phi_k(G_{N_{r+1},S})-\Phi_k(G_{N_{r},S}) > \lambda,
\end{equation}
then there is a tuple  $(d_1,\dots, d_k)$ with 
\begin{equation}\label{eq:distance-d}
 0\leq d_1 \leq \dots \leq d_k < N_{r+1},
\end{equation}
an integer $m$ with
\begin{equation}\label{eq:distance-m}
 N_{r}-d_{k}+1\leq m\leq N_{r+1}-d_{k},
\end{equation}
and $1\leq s_1,\dots, s_{k}\leq S$ such that
\begin{equation}\label{eq:distance-1}
 \left|\sum_{n=\max\{1, N_{r}-d_k+1\}}^{m}e_{n+d_1}(s_1)\dots e_{n+d_k}(s_k) \right|>\lambda.
\end{equation}
By Lemmas \ref{lemma:independence} and  \ref{lemma:Spm} we have that the probability, that \eqref{eq:distance-1} holds is at most
\begin{equation*}
 \exp\left\{-\frac{\lambda^2}{2(N_{r+1}-N_{r})}\right\}\leq N_{r+1}^{-3k+3}.
\end{equation*}
Summing up all possible tuples $(d_1,\dots, d_k)$, all possible integers $m$ and all possible $1\leq s_1,\dots, s_{k}\leq S$, the probability that \eqref{eq:distance-1} happens for some  $(d_1,\dots, d_k)$ satisfying \eqref{eq:distance-d}, some $m$ satisfying \eqref{eq:distance-m} and some $1\leq s_1,\dots, s_{k}\leq S$ is at most
\begin{equation*}
 (N_{r+1}-N_r)(N_{r+1}S)^kN_{r+1}^{-3k+3}leq N_{r+1}^{-2k+4}S^{k}.
\end{equation*}
This is also an upper bound for the probability of \eqref{eq:distance-cel}, and so
\[
 \P[\Phi_k(G_{N_{r+1}})-\Phi_k(G_{N_{r}}) > \lambda]\leq N_{r+1}^{-2k+4}S^{k}.
\]
Summing it over all $r$ we get
\[
 \sum_{r=1}^\infty\frac{1}{2k-1}N_{r+1}^{-2k+4}S^{k} <\infty
\]
and the result follows from the Borel-Cantelli Lemma.
\end{proof}

\begin{proof}[Theorem \ref{thm:limit}]
Write
\[
 \lambda (N)=\sqrt{2kN\log N }.
\]
Let $N_r=\lceil e^{r^{1/2}} \rceil$ ($r=1,2,\dots$) be a sequence of integers.
We remark, that 
\[
 \lim_{k\rightarrow \infty}\frac{N_{r+1}}{N_{r}}=\lim_{k\rightarrow \infty}e^{(r+1)^{1/2}-r^{1/2}}=1.
\]

First we prove, that
\begin{equation}\label{eq:subseq}
 \max_{N_{r-1}\leq N\leq N_{r}} \left| \frac{\Phi_k\left(G_{N_r,S}\right)}{\lambda(N_r)} -1\right|\rightarrow 0 \quad \text{almost surely}.
\end{equation}
For any $\varepsilon >0$ we have
\begin{align*}
 &\P\left[\left| \frac{\Phi_k\left(G_{N_r,S}\right)}{\lambda(N_r)} -1\right|>\varepsilon \right]\\
 &\leq 
 \P\left[\left| \frac{\E[\Phi_k\left(G_{N_r,S}\right)]}{\lambda(N_r)} -1\right|>\frac{\varepsilon}{2} \right]+
 \P\left[\left| \frac{\Phi_k\left(G_{N_r,S}\right)}{\lambda(N_r)} -\frac{\E[\Phi_k\left(G_{N_r,S}\right)]}{\lambda(N_r)} \right|>\frac{\varepsilon}{2} \right].
\end{align*}
The first term equals zero for sufficiently large $r$ by Lemma \ref{lemma:conv-E}. By Lemma \ref{lemma:distance} the second term is at most
\begin{align*}
 &2\exp \left\{ -\frac{\varepsilon ^2 \lambda(N_r)^2}{8k^2N_r}\right\}\leq  \exp \left\{ -\frac{\varepsilon ^2 \log N_r }{4k}\right\} \leq  \exp \left\{ -\frac{\varepsilon ^2 \log N_r }{4k}\right\}.
\end{align*}
Applying a crude estimate we get that for sufficiently large $r$
\[
 \exp \left\{ -\frac{\varepsilon ^2 \log N_r }{4k}\right\}\leq \frac{1}{4k\log N_r ^3}\leq r^{-3/2}.
\]
Thus for a sufficiently large $r_0$ we have
\[
 \sum_{r=r_0}^{\infty} \P\left[\left| \frac{\Phi_k(G_{N_r,S})}{\lambda(N_r)} -1\right|>\varepsilon \right]\leq  \sum_{r=r_0}^{\infty} r^{-3/2}<\infty
\]
and \eqref{eq:subseq} follows from the Borel-Cantelli Lemma.

Next we show
\begin{equation*}
  \max_{N_r\leq N \leq N_{r+1}}\left|\frac{\Phi_k\left(G_{N,S}\right)}{\lambda(N)}-1\right|\rightarrow 0 \quad \text{almost surely}.
\end{equation*}

By the triangle inequality we have
\begin{align}\label{eq:final-1}
 &\max_{N_r\leq N \leq N_{r+1}}\left|\frac{\Phi_k\left(G_{N,S}\right)}{\lambda(N)}-1\right| \notag\\
& \leq 
 \max_{N_r\leq N \leq N_{r+1}} \left|\frac{\Phi_k\left(G_{N_{r+1},S}\right)}{\lambda(N_{r+1})}-1\  \right| \notag 
 +
 \max_{N_r\leq N \leq N_{r+1}}\left|\frac{\Phi_k\left(G_{N_{r+1},S}\right)}{\lambda(N_{r+1})}- \frac{\Phi_k\left(G_{N,S}\right)}{\lambda(N_{r+1})}  \right|\\
 &+
  \max_{N_r\leq N \leq N_{r+1}}\left|
  \frac{\Phi_k\left(G_{N,S}\right)}{\lambda(N)}- \frac{\Phi_k\left(G_{N,S}\right)}{\lambda(N_{r+1})}  
  \right|.
\end{align}

The first term goes to zero almost surely by \eqref{eq:subseq}. Since $\Phi_k(G_{N,S})$ is non-decreasing in $N$, we have by Lemma \ref{lemma:speed} that for sufficiently large $r$
\begin{equation}\label{eq:final-2}
  \max_{N_r\leq N \leq N_{r+1}}\left|\frac{\Phi_k\left(G_{N_{r+1},S}\right)}{\lambda(N_{r+1})}- \frac{\Phi_k\left(G_{N,S}\right)}{\lambda(N_{r+1})}  \right|\leq \sqrt{3\frac{N_{r+1}-N_r}{N_{r+1}}}\rightarrow 0.
\end{equation}
On the other hand 
\begin{multline}\label{eq:final-3}
 \max_{N_r\leq N \leq N_{r+1}}\left|
 \frac{\Phi_k\left(G_{N,S}\right)}{\lambda(N)}- \frac{\Phi_k\left(G_{N,S}\right)}{\lambda(N_{r+1})}  
  \right|\\
  \leq \frac{\Phi_k\left(G_{N_{r+1}}\right)}{\lambda(N_{r+1})} \max_{N_r\leq N \leq N_{r+1}} \left|\frac{\lambda(N_{r+1})}{\lambda(N_{r})} -1\right|\rightarrow 0, \quad \text{almost surely}
\end{multline}
by \eqref{eq:subseq}.

Finally, the result follows from \eqref{eq:subseq}, \eqref{eq:final-1}, \eqref{eq:final-2} and \eqref{eq:final-3}.
\end{proof}

\section{Minimal values of $\Phi_k (G)$}

First we prove Theorem \ref{thm:min-odd}.

\begin{proof}[Theorem \ref{thm:min-odd}]
For the lower bound we can assume, that $4k+1<S$ otherwise the bound is trivial. Consider the maximal integer $L$ such that
 \[
  2k \cdot 2^L+1\leq S.
 \]
By the pigeon hole principle, there are different numbers $1\leq s_1<s_2<\dots <s_{2k+1}\leq S$ such that their first $L$ elements are coincide:
\[
 e_n(s_1)=e_n(s_2)=\dots =e_n(s_{2k+1}), \quad n=1,2,\dots, L.
\]
Thus
\[
 \Phi_{2k+1} \left(G_{N,S}\right)\geq \left| \sum_{n=1}^L e_n(s_1)e_n(s_2)\dots e_n(s_{2k+1}) \right|=L=\left\lfloor \log_2 \frac{S-1}{2k} \right\rfloor.
\]

For the upper bound, given $S$ we construct a map $G_{N,S}$
with small cross-correlation measure $\Phi_{2k+1}\left(G_{N,S}\right)$.

Let $K=\lceil \log_2 S\rceil$ and define the sequences by
\[
 e_n(i)=
 \left\{
 \begin{array}{ll}
  (-1)^{i_n}, & 1\leq n< K, \\
  (-1)^{i_{K}+n} & K \leq n \leq N, 
 \end{array}
 \right.
 \quad i=1,2,\dots, S
\]
where $i_1,i_2,\dots,i_{K}\in\{0,1\}$ are the binary digits of $i-1$
\[
 i-1=\sum_{n=1}^{K}i_n2^{n-1}.
\]

Then for $M\leq N$, all $0\leq d_1\leq \dots \leq d_{2k+1}\leq N-M$ and all $1\leq s_1,s_2,\dots , s_{2k+1}\leq S$ with $d_i\neq d_j$ if $s_i=s_j$ we have
\begin{equation*}
\begin{split}
&\left|\sum_{n=1}^M e_{n+d_1}(s_1)e_{n+d_2}(s_2)\dots e_{n+d_{2k+1}}(s_{2k+1})\right|\\
&\leq \left|\sum_{n=1}^{K-1} e_{n+d_1}(s_1)e_{n+d_2}(s_2)\dots e_{n+d_{2k+1}}(s_{2k+1})\right|\notag \\
&\qquad \qquad+ \left|\sum_{n=K}^M e_{n+d_1}(s_1)e_{n+d_2}(s_2)\dots e_{n+d_{2k+1}}(s_{2k+1})\right|\\
 &\leq  \, K-1 + \left|\sum_{n=K}^M (-1)^n\right|\leq K
\end{split}
\end{equation*}
which proves the upper bound.
\end{proof}

The proof of Theorem \ref{thm:min-even} is similar to the proof of \eqref{eq:c_k-min}, it is based on the following lemmas (Lemmas 5 and 6 in \cite{AKMMR-min}).

\begin{lemma}\label{lemma:1}
 For a symmetric matrix $\A=(A_{i,j})_{1\leq 1,j\leq n}$ with $A_{i,i}=1$ for all $i$ and $|A_{i,j}|\leq \varepsilon$ for all $i\neq j$ we have
 \[
\rank (\A)\geq \frac{n}{1+\varepsilon^2 (n-1)}.
 \]
\end{lemma}

\begin{lemma}\label{lemma:2}
 Let $\A=(A_{i,j})_{1\leq 1,j\leq n}$ be a real matrix with $A_{i,i}=1$ for all $i$ and $|A_{i,j}|\leq \varepsilon$ for all $i\neq j$, where $\sqrt{1/n}\leq \varepsilon \leq 1/2$. Then
 \begin{equation}\label{eq:lemma-2}
\rank (\A)\geq \frac{1}{100\varepsilon ^2 \log (1/\varepsilon )}\log n.
 \end{equation}
 If $\A$ is symmetric, then (\ref{eq:lemma-2}) holds with the constant $1/100$ replaced by $1/50$.
\end{lemma}

\begin{proof}[Theorem \ref{thm:min-even}]
First consider the case when $S$ is large $2kN\leq S$.
Let $t=\lfloor S/k \rfloor$ and let $L_1,L_2,\dots, \allowbreak L_t\subset \{1,2,\dots, S\}$ be distinct subsets with $k$ elements. For each $L_i$ ($i=1,2,\dots, t$) we assign the vector $\v_i\in\mathbb{R}^N$ with
 \[
  \v_{i,j}=\prod_{s\in L_i}e_j(s).
 \]
Define the matrix $\A=(A_{i,j})_{1\leq i,j \leq t}$ by
\[
 A_{i,j}=\frac{1}{N}\langle \v_i,\v_j \rangle=\frac{1}{N}\sum_{n=1}^N \prod_{s\in L_i \cup L_j} e_{n}(s),
\]
where $\langle \cdot, \cdot\rangle$ is the standard inner-product. Then we have
\[
 A_{i,i}=1, \quad i=1,2,\dots, t, 
\]
and
\begin{equation}\label{eq:thm1-0}
  \Phi_{2k}\left(G_{N,S}\right)\geq N \max \{|A_{i,j}|: \ 1\leq i,j \leq t, \ i\neq j \}.
\end{equation}

Let $\B=(\v_i^T)_{1\leq i \leq t}$ be the $t \times N$ matrix with rows $\v_i^T$ ($1\leq i \leq t$). Clearly $\A=N^{-1}\B\B^T$, thus $\rank (\A)\leq N$. On the other hand, by Lemma \ref{lemma:1} we have
\begin{equation}\label{eq:thm1-1}
 N\geq \rank(\A)\geq \frac{t}{1+\varepsilon ^2 (t-1)},
\end{equation}
where $\varepsilon =\max \{|A_{i,j}|: \ 1\leq i,j \leq t, \ i\neq j \}$. (\ref{eq:thm1-1}) gives
\begin{equation}\label{eq:thm1-2}
 \varepsilon^2 \geq \frac{1}{N} - \frac{1}{t}\geq \frac{1}{t}. 
\end{equation}

Now we are ready to complete the proof of the result. If $\varepsilon >1/2$, then the theorem trivially holds. Otherwise, by (\ref{eq:thm1-2}) we have $\sqrt{1/t}\leq \varepsilon \leq 1/2$, thus by Lemma \ref{lemma:2} we have
\begin{equation*}
 N\geq \rank(\A)\geq \frac{1}{50 \varepsilon^2 \log (1/\varepsilon)} \log \left\lfloor \frac{S}{k}\right\rfloor,
\end{equation*}
whence
\begin{equation}\label{eq:epsilon-1}
  \varepsilon^2 \log (1/\varepsilon)\geq \frac{1}{50 N} \log \left\lfloor \frac{S}{k}\right\rfloor.
\end{equation}
Using that $1/\varepsilon \geq \log 1/\varepsilon$, we have from \eqref{eq:epsilon-1}, that
\begin{equation}\label{eq:epsilon-2}
 \varepsilon \geq  \varepsilon^2 \log (1/\varepsilon)\geq \frac{1}{50 N} \log \left\lfloor \frac{S}{k}\right\rfloor.
\end{equation}
Thus from \eqref{eq:epsilon-1} and \eqref{eq:epsilon-2} we get
\begin{equation*}
 \varepsilon^2 \log \frac{50 N}{\log \lfloor S/k \rfloor}\geq \frac{1}{50 N} \log \left\lfloor \frac{S}{k}\right\rfloor,
\end{equation*}
and hence
\begin{equation}\label{eq:thm1-f}
  \varepsilon \geq \sqrt{ \frac{\log \lfloor S/k \rfloor}{50 N} \left/ { \log \frac{50 N}{\log \lfloor S/k \rfloor}  } \right. }.
\end{equation}
Then \eqref{eq:min-1} follows from \eqref{eq:thm1-0} and \eqref{eq:thm1-f}.

Next, consider the case when $S$ is small $2kN>S$. Put $\ell=\left\lceil{k}/{S} \right\rceil$ and write $M=\lfloor N/(2\ell +1)\rfloor$ and $N'=N-M+1$. Let $L_1,\dots, L_t$ 
pairwise disjoint $k$-element subsets of $\{1,\dots, S\}\times \{1,\dots, N'\}$ with
\[
 t\geq S\cdot  \left\lfloor\frac{N'}{\ell}\right\rfloor.
\]
Then
\[
  t\geq S\cdot  \left\lfloor\frac{N-\lfloor N/(2\ell +1) \rfloor+1}{\ell}\right\rfloor\geq 2M.
\]
In the same way as before, we get from Lemma \ref{lemma:1} that
\[
 \Phi_{2k}\geq M \sqrt{\frac{1}{M}-\frac{1}{t}}\geq \sqrt{M-\frac{M^2}{t}}\geq \sqrt{M-\frac{M}{2}}\geq \sqrt{\frac{N}{2\lceil k/S \rceil+1}}
\]
and \eqref{eq:min-2} follows.
\end{proof}

\subsection*{Acknowledgment}
The author would like to thank Andr\'as S\'ark\"ozy for helpful discussions. The author also thank the anonymous referees for their valuable comments.

The author is partially supported by the Austrian Science Fund FWF Project F5511-N26 
which is part of the Special Research Program "Quasi-Monte Carlo Methods: Theory and Applications" and by Hungarian National Foundation for Scientific Research,
Grant No. K100291.



\end{document}